\documentclass[12pt, reqno]{amsart}
\usepackage{ amsmath,amsthm, amscd, amsfonts, amssymb, graphicx, color}
\usepackage[bookmarksnumbered, colorlinks, plainpages]{hyperref}
\newtheorem{thm}{Theorem}[section]
\newtheorem{cor}[thm]{Corollary}
\newtheorem{lem}[thm]{Lemma}

\newtheorem{exam}[thm]{Example}
\numberwithin{equation}{section}

\begin{document}

\title[the sum of a tripotent and a nilpotent]{Rings over which every matrix is the sum of a tripotent and a nilpotent}

\author{Marjan Sheibani}
\author{Huanyin Chen}
\address{
Faculty of Mathematics\\Statistics and Computer Science\\  Semnan University\\ Semnan, Iran}
\email{<m.sheibani1@gmail.com>}
\address{
Department of Mathematics\\ Hangzhou Normal University\\ Hang -zhou, China}
\email{<huanyinchen@aliyun.com>}


\begin{abstract}
A ring $R$ is  trinil clean if every element in $R$ is the sum of a tripotent and a nilpotent.
If $R$ is a 2-primal strongly 2-nil-clean ring, we prove that $M_n(R)$ is trinil clean for all $n\in {\Bbb N}$.
Furthermore, we show that the matrix ring over a strongly 2-nil-clean ring of bounded index is trinil clean. We thereby provide various type of rings over which every matrix is the sum of a tripotent and a nilpotent.\end{abstract}

\maketitle

\section{Introduction}

Throughout, all rings are associative with an identity. A ring $R$ is nil clean provided that every element in $R$ is the sum of an idempotent and a nilpotent ~\cite{D}. A ring $R$ is weakly nil-clean provided that every element in $R$ is the sum or difference of a nilpotent and an idempotent~\cite{BD}.
The subjects of nil-clean rings and weakly nil-clean rings are interested for so many mathematicians, e.g., ~\cite{BD,B,CH,DW,D,K,KWZ,KZ} and ~\cite{ST}.

The purpose of this paper is to consider matrices over a new type of rings which cover (weakly) nil clean rings. An element $e\in R$ is a tripotent if $e=e^3$. We call a ring $R$ is trinil clean provided that every element in $R$ is the sum of a tripotent and a nilpotent. We shall explore when a matrix ring is trinil-clean, i.e., when every matrix over a ring can be written as the sum of a tripotent and a nilpotent. A ring $R$ is 2-primal if its prime radical coincides with the set of nilpotent elements of the ring, e.g., commutative rings, reduced rings, etc.
Following the authors, a ring $R$ is strongly 2-nil-clean if every element in $R$ is the sum of two idempotents and a nilpotent that commute. Evidently, a ring $R$ is strongly 2-nil-clean if and only if every element in $R$ is the sum of a tripotent and a nilpotent that commute (see~\cite[Theorem 2.8]{CS}). If $R$ is a 2-primal strongly 2-nil-clean ring, we prove that $M_n(R)$ is trinil clean for all $n\in {\Bbb N}$.
Furthermore, we show that the matrix ring over a strongly 2-nil-clean ring of bounded index is trinil clean. This provides a large class of rings over which
every matrix is the sum of a tripotent and a nilpotent.

We use $N(R)$ to denote the set of all nilpotent elements in $R$ and $J(R)$ the Jacobson radical of $R$.
${\Bbb N}$ stands for the set of all natural numbers.

\section{Structure Theorems}

\vskip4mm The aim of this section is to investigate general structure of trinil clean rings which will be used in the sequel.
We begin several examples of such rings.

\begin{exam} The class of trinil clean rings contains many familiar examples.\end{exam}
\begin{enumerate}
\item [(1)] Every weakly nil-clean ring is trinil clean, e.g., strongly nil-clean rings, nil-clean rings, Boolean rings, weakly Boolean rings.
\item [(2)] Every strongly trinil clean ring is trinil clean.
\item [(3)] A local ring $R$ is trinil clean if and only if $R/J(R)\cong {\Bbb Z}_2$ or ${\Bbb Z}_3$, $J(R)$ is nil.
\end{enumerate}
\vspace{-.5mm}

We also provide some examples illustrating which ring-theoretic extensions of trinil clean rings produce trinil clean rings.

\begin{exam} \end{exam}
\begin{enumerate}
\item [(1)] Any quotient of a trinil clean ring is trinil clean.
\item [(2)] Any finite product of trinil clean rings is trinil clean. But $R={\Bbb Z}_2\times {\Bbb Z}_4\times {\Bbb Z}_8\times $ is an infinite product of trinil clean rings, which is not trinil clean. The element $(0,2,2,2,\cdots )\in R$ can not written as the sum of a tripotent and a nilpotent element.
\item [(3)] The triangular matrix ring $T_n(R)$ is trinil clean if and only if $R$ is trinil clean.
\item [(4)] The quotient ring $R[[x]]/(x^n) (n\in {\Bbb N})$ of a trinil clean ring $R$ is trinil clean.
\end{enumerate}

\begin{lem} Let $R$ be trinil clean. Then $6\in N(R)$.\end{lem}
\begin{proof} By hypothesis, there exists a tripotent $e\in R$ such that $w:=2-e\in N(R)$ with $ew=we$, and so
$2-2^3=(e+w)-(e+w)^3\in N(R)$. This shows that $6\in N(R)$.\end{proof}

\begin{thm} Let
$R$ be a ring. Then the following are equivalent:
\end{thm}
\begin{enumerate}
\item [(1)]{\it $R$ is trinil clean.}
\vspace{-.5mm}
\item [(2)]{\it $R=A\times B$, where $A$ and $B$ are trinil clean, $2\in J(A)$ and $3\in J(B)$.}
\end{enumerate}
\begin{proof} $(1)\Rightarrow (2)$ In view of Lemma 2.3, $6\in N(R)$.
Write $6^n=0~(n\in {\Bbb N})$. Since $2^nR+3^nR=R$ and $2^nR\bigcap 3^nR=0$. By Chinese Remainder Theorem, we have
$R\cong A\times B$, where $A=R/2^nR$ and $B=R/3^nR$. Clearly, $2\in N(A)$ and $3\in N(B)$. This implies that
$2\in J(A)$ and $3\in J(B)$, as desired.

$(2)\Rightarrow (1)$ This is obvious, by Example 2.2 (2).\end{proof}

\begin{lem} Let $R$ be trinil clean. Then $J(R)$ is nil.\end{lem}
\begin{proof} Let $x\in J(R)$. Then $x=a+b, a\in N(R), b=b^3$. Suppose $a^q=0$. Then
$(x-b)^{2q+1}=a^{2q+1}=0$; hence, $b^{2q+1}\in J(R)$. Clearly, $b^2\in R$ is an idempotent. Hence, $b^{2q+1}=b(b^2)^q=b(b^2)=b^3=b$. It follows that $b\in J(R)$, and so $b(1-b^2)=0$. This implies that $b=0$; hence, $x=a\in N(R)$. Accordingly, $J(R)$ is nil, as required.\end{proof}

\begin{lem} \cite[Lemma 3.5]{KW} Let $R$ be a ring, let $a\in R$. If $a^2-a\in N(R)$, then there exists a monic polynomial $f(t)\in {\Bbb Z}[t]$ such that
$f(a)^2=f(a)$ and $a-f(a)\in N(R)$.\end{lem}

\begin{thm} Let $R$ be a ring and $3\in J(R)$. Then $R$ is trinil clean if and only if
\end{thm}
\begin{enumerate}
\item [(1)]{\it $R/J(R)$ is trinil clean;}
\vspace{-.5mm}
\item [(2)]{\it $J(R)$ is nil.}
\end{enumerate}
\begin{proof} $\Longrightarrow $ This is obvious, by Example 2.2 (1) and Lemma 2.6.

$\Longleftarrow $ Let $a\in R$. Then there exists some $\overline{f^3}=\overline{f}\in R/J(R)$ such that $\overline{a-f}\in N(R/J(R))$.
As $J(R)$ is nil, we see that $w:=a-f\in N(R)$.
Let $e=1-f$. Then $e^3-e=(1-f)^3-(1-f)=-3f+3f^2-(f^3-f)\equiv 0~(mod J(R))$.
We check that $(-2e^2)^2=4e^4\equiv -2e^4\equiv -2e^2 ~(mod J(R))$ as $6\in J(R)$.
Moreover, we have $(e+2e^2)^2\equiv e^2+4e+4e^2=e+2e^2+3(e-e^2)+6e^2\equiv e+2e^2~(mod J(R))$.
In light of Lemma 2.6, we can find idempotents $f(e),g(e)\in R$ such that
$r:=-2e^2-f(e), s:=e+2e^2-g(e)\in J(R)$. Here, $f(t),g(t)\in {\Bbb Z}[t]$. Then
$e=(-2e^2)+(e+2e^2)=f(e)+r+g(e)+s$. Hence,
$a=1-e+w=(1-f(e))-g(e)+w-r-s$.
Clearly, $(1-f(e))g(e)=g(e)(1-f(e))$, we see that
$((1-f(e))-g(e))^3=(1-f(e))-g(e)$ and $w-r-s\in N(R)$.
Therefore $R$ is trinil clean, as asserted.\end{proof}

\section{Strongly 2-Nil Clean Rings}

In this section, We shall establish the connections between trinil clean rings and strongly 2-nil-clean rings. We have

\begin{lem} (see~\cite[Theorem 2.8]{CS}). Let $R$
be a ring. Then the following are equivalent:
\end{lem}
\begin{enumerate}
\item [(1)] {\it $R$ is strongly 2-nil-clean.}
\vspace{-.5mm}
\item [(2)] {\it For all $a\in R$, $a-a^3\in N(R)$.}
\vspace{-.5mm}
\item [(3)] {\it Every element in $R$ is the sum of two idempotents and a nilpotent that commute.}
\end{enumerate}

Recall that a ring $R$ is right (left) quasi-duo if every right (left) maximal ideal of $R$ is an ideal. We come now to the following.

\begin{thm} A ring $R$ is strongly 2-nil-clean if and only if
\end{thm}
\begin{enumerate}
\item [(1)]{\it $R$ is trinil clean;}
\vspace{-.5mm}
\item [(2)]{\it $R$ is right (left) quasi-duo;}
\vspace{-.5mm}
\item [(3)]{\it $J(R)$ is nil.}
\end{enumerate}
\begin{proof} $\Longrightarrow$ $(1)$ is obvious. As in the proof of Lemma 3.1, $a^3-a\in N(R)$. It follows by ~\cite[Theorem A1]{HT1} that $N(R)$ forms an ideal of $R$.
Hence, $N(R)\subseteq J(R)$, and then $R/J(R)$ is tripotent. In view of ~\cite[Theorem 1]{HT}, $R/J(R)$ is commutative.Let $M$ be a right (left) maximal ideal of $R$.Then $M/J(R)$ is an ideal of $R/J(R)$. Let $x\in M, r\in R$. Then $\overline{rx}\in M/J(R)$, and then $rx\in M+J(R)\subseteq M$. This shows that $M$ is an ideal of $R$. Therefore $R$ is right (left) quasi-duo. $(3)$ is follows from Lemma 2.5.

$\Longleftarrow$ Since $R$ is trinil clean, so is $R/J(R)$. As $R$ is right (left) quasi-duo, we can see in the same way as ~\cite[Theorem 2.8]{CH} that $R/J(R)$ is abelian, and so it is strongly 2-nil-clean. Let $x\in R$. Then $\overline{x^3-x}\in N(R/J(R))$  by Lemma 3.1. As $J(R)$ is nil, we see that $x^3-x\in N(R)$. By using Lemma 3.1 again, $R$ is strongly 2-nil-clean.\end{proof}.

\begin{cor} A ring $R$ is strongly 2-nil-clean if and only if
\end{cor}
\begin{enumerate}
\item [(1)]{\it $R/J(R)$ is tripotent;}
\vspace{-.5mm}
\item [(2)]{\it $J(R)$ is nil.}
\end{enumerate}

\vskip4mm Recall that a ring $R$ is NI if $N(R)$ forms an ideal of $R$. We now derive

\begin{thm} A ring $R$ is strongly 2-nil-clean if and only if
\end{thm}
\begin{enumerate}
\item [(1)]{\it $R$ is trinil clean;}
\vspace{-.5mm}
\item [(2)]{\it $R$ is NI.}
\end{enumerate}
\begin{proof} $\Longrightarrow $ Clearly, $R$ is trinil clean. Let $a\in R$. Then $a=e+w$ for some tripotent $e$ and a nilpotent $w$ that commute. Hence, $a-a^3=w(3e^2+3ew+w^2-1)\in N(R)$, so by ~\cite[ Theorem A1]{HT}, $N(R)$ is an ideal of $R$. That is, $R$ is NI.

$\Longleftarrow$ Let $a\in R$. Then $a=e+w$ where $e^3=e\in R$ and $w\in N(R)$. hence, $a^3-a\in N(R)$. In light of Lemma 3.1,
$R$ is strongly 2-nil-clean, as asserted.\end{proof}

\begin{cor} Every 2-primal trinil clean ring is strongly 2-nil-clean.
\end{cor}
\begin{proof} As every 2-primal ring is NI, the result follows from Theorem 3.4.\end{proof}

\begin{thm} Let $R$ be a ring. If $N(R)$ is commutative, then $R$ is strongly 2-nil-clean if and only if $R$ is trinil clean.
\end{thm}
\begin{proof} $\Longrightarrow$ This is trivial.

$\Longleftarrow$ Let $e^2=e\in R$
and $r\in N(R)$. Write $r^{2^n}=0$ for some $n\in {\Bbb N}$. Since
$\big(er(1-e)\big)^2=\big((1-e)re\big)^2=0$, we see that $er(1-e),
(1-e)re\in N(R)$. Hence,
$er^2e-(ere)^2=er(1-e)re=e\big(er(1-e)\big)\big(1-e)re\big)=e\big(1-e)re\big)\big(er(1-e)\big)=0$.
Thus, $er^4e=e(r^2)^2e=\big(er^2e\big)^2=(ere)^4$. Repeating this
procedure, $(ere)^{2^n}=er^{2^n}e=0$, and so $ere\in N(R)$.
Consequently, $E(R)N(R)\subseteq N(R)$.

Suppose that $x=x^3$. Let $r\in N(R)$. Then we have some $m\in {\Bbb N}$
such that $x^m\in E(R)$, and suppose that $m\geq 2$. By the
previous claim, $x^mr\in N(R)$. Write $\big(x^mr\big)^k=0 (k\in
{\Bbb N})$. Then
$\big(x^irx^{m-i}\big)^{k+1}=x^ir\big(x^mr\big)^kx^{m-i}=0$;
hence, $x^irx^{m-i}\in N(R)$ for any $0\leq i\leq m$. If $m=2n$,
then $x^nrx^n\in N(R)$, and so $x^nr\in N(R)$. If $m=2n+1$, then
$$b:=\big(xrx^{2n}\big)\big(x^2rx^{2n-1}\big)\cdots
\big(x^{2n+1}r\big)\in N(R).$$ This shows that
$x^mb=\big(x^{m+1}r\big)^m$. As $b\in N(R)$, we get $x^mb\in
N(R)$, and then $\big(x^{m+1}r\big)^m\in N(R)$. Hence,
$x^{m+1}r\in N(R)$. In any case, we can find some $m'<m$ such that
$x^{m'}N(R)\subseteq N(R)$. By iteration of this process, we get
$xN(R)\subseteq N(R)$. That is, $xN(R)\subseteq N(R)$ for any potent $x\in R$.
Analogously, we deduce that $N(R)x\subseteq N(R)$ for any potent $x\in R$. Therefore $R$ is NI.
According to Theorem 3.4, $R$ is strongly 2-nil-clean.\end{proof}

A natural problem is if the matrix ring over a strongly 2-nil-clean ring is strongly 2-nil-clean. The answer is negative as the following shows.

\begin{exam} Let $n\geq 2$. then matrix ring $M_n(R)$ is not strongly 2-nil-clean for any ring $R$.
\end{exam}
\begin{proof} Let $R$ be a ring, and let $A=\left(
\begin{array}{cc}
1_R&1_R\\
1_R&0
\end{array}
\right)$. Then $A^3-A=\left(
\begin{array}{cc}
2&1_R\\
1_R&1_R
\end{array}
\right)$. One checks that $\left(
\begin{array}{cc}
2&1_R\\
1_R&1_R
\end{array}
\right)^{-1}=\left(
\begin{array}{cc}
1_R&-1_R\\
-1_R&2
\end{array}
\right)$, and so $A^3-A$ is not nilpotent. If $M_n(R)$ is strongly 2-nil-clean, it follows by Lemma 3.1 that $A^3-A$ is nilpotent, a contradiction, and we are done.\end{proof}

\section{Trinil Clean Matrix Rings}

The purpose of this section is to investigate when a matrix ring over a strongly 2-nil-clean is trinil clean. We now extend ~\cite[Theorem 3]{B} and ~\cite[Theorem 20]{BD} as follows.

\begin{thm} Let $K$ be a field. Then the following are equivalent:
\end{thm}
\begin{enumerate}
\item [(1)]{\it $M_n(K)$ is trinil clean.}
\vspace{-.5mm}
\item [(2)]{\it $K\cong {\Bbb Z}_2$ or ${\Bbb Z}_3$.}
\end{enumerate}
\begin{proof} $(1)\Rightarrow (2)$ Let $0\neq a\in K$. Choose $A=aI_n$. Then $aI_n=E+W$, where
$E^{3}=E$ and $W\in M_n(K)$ is nilpotent. Clearly,
$E=aI_n\big(I_n-a^{-1}W\big)\in GL_n(K)$, and so $E^2=I_n$. As $K$
is commutative, we see that $EW=(aI_n-W)W=W(aI_n-W)=WE$. From
this, we get $a^2I_n-E^2\in M_n(K)$ is nilpotent; hence, $a^2-1\in
K$ is nilpotent. This shows that $a^2=1$. Thus, $(a-1)(a+1)=0$; whence $a=1$ or $-1$. If $1=-1$,
then $K\cong {\Bbb Z}_2$. If $1\neq -1$, then $K\cong {\Bbb Z}_3$, as required.

$(2)\Rightarrow (1)$ As every
matrix over a field has a Frobenius normal form, and that trinil clean matrix is invariant under the similarity, we may assume
that $$A=\left(
\begin{array}{cccccc}
0&&&&&c_0\\
1&0&&&&c_1\\
&1&0&&&c_2\\
&&&\ddots&&\vdots\\
&&&\ddots&0&c_{n-2}\\
&&&&1&c_{n-1}
\end{array}
\right).$$

Case I. $c_{n-1} =1$, we claim that there exist some $E=E^2$  such that $A-E\in M_n(K)$ is nilpotent. If $c_{n-1}=1$, Choose $$W=\left(
\begin{array}{cccccc}
0&&&&&0\\
1&0&&&&0\\
&1&0&&&0\\
&&&\ddots&&\vdots\\
&&&\ddots&0&0\\
&&&&1&0
\end{array}
\right), E=\left(
\begin{array}{cccccc}
0&&&&&c_0\\
0&0&&&&c_1\\
&0&0&&&c_2\\
&&&\ddots&&\vdots\\
&&&\ddots&0&c_{n-2}\\
&&&&0&1
\end{array}
\right).$$ Then $E^2=E$, and so $A=E+W$ is trinil clean.

Case II. $c_{n-1}=-1$, choose $$W=\left(
\begin{array}{cccccc}
0&&&&&0\\
1&0&&&&0\\
&1&0&&&0\\
&&&\ddots&&\vdots\\
&&&\ddots&0&0\\
&&&&1&0
\end{array}
\right), E=\left(
\begin{array}{cccccc}
0&&&&&c_0\\
0&0&&&&c_1\\
&0&0&&&c_2\\
&&&\ddots&&\vdots\\
&&&\ddots&0&c_{n-2}\\
&&&&0&-1
\end{array}
\right).$$ Then $E^2=-E$, so $E^3=E$ and $A=E+W$.

Case III. $c_{n-1}=0$.

If $n=2$, then $$\left(
\begin{array}{cc}
0&c_0\\
1&0
\end{array}
\right)=\left(
\begin{array}{cc}
0&1\\
1&0
\end{array}
\right)+\left(
\begin{array}{cc}
0&c_0-1\\
0&0
\end{array}
\right).$$

If $n=3$, then $$\left(
\begin{array}{ccc}
0&0&c_0\\
1&0&c_1\\
0&1&0
\end{array}
\right)=\left(
\begin{array}{ccc}
0&0&0\\
1&0&1\\
0&1&0
\end{array}
\right)+\left(
\begin{array}{ccc}
0&0&c_0\\
0&0&c_1-1\\
0&0&0
\end{array}
\right).$$

If $n\geq 4$, we have
$$\begin{array}{l}
A=\left(
\begin{array}{cccccccc}
0&0&0&\cdots&0&0&0&0\\
0&0&0&\cdots&0&0&0&0\\
0&0&0&\cdots&0&0&0&0\\
\vdots&\vdots&\vdots&\ddots&\vdots&\vdots&\vdots&\vdots\\
0&0&0&\cdots&0&0&0&0\\
0&0&0&\cdots&0&1&0&1\\
0&0&0&\cdots&0&0&1&0
\end{array}
\right)\\
+\left(
\begin{array}{cccccccc}
0&0&0&\cdots&0&0&0&c_0\\
1&0&0&\cdots&0&0&0&c_1\\
0&1&0&\cdots&0&0&0&c_2\\
\vdots&\vdots&\vdots&\ddots&\vdots&\vdots&\vdots&\vdots\\
0&0&0&\cdots&1&0&0&c_{n-3}\\
0&0&0&\cdots&0&0&0&c_{n-2}-1\\
0&0&0&\cdots&0&0&0&0
\end{array}
\right).
\end{array}$$ This implies that $A=E+W$ where $E=E^3$ and $W\in M_n({\Bbb Z}_3)$ is nilpotent. This completes the proof.\end{proof}

\begin{thm} Let $R$ be tripotent. Then $M_n(R)$ is trinil clean for all $n\in {\Bbb N}$.\end{thm}
\begin{proof} Let $A\in M_n(R)$, and
let $S$ be the subring of $R$ generated by the entries of $A$.
That is, $S$ is formed by finite sums of monomials of the form:
$a_1a_2\cdots a_m$, where $a_1,\cdots,a_m$ are entries of $A$. Since $R$ is a
commutative ring in which $6=0$, $S$
is a finite ring in which $x=x^3$ for all $x\in S$. Thus, $S$ is isomorphic to finite direct
product of ${\Bbb Z}_2$ or ${\Bbb Z}_3$. Clearly, ${\Bbb Z}_2$ and ${\Bbb Z}_3$ are tripotent. Hence, $S$ is tripotent. Thus, $M_n(S)$ is
trinil clean. As $A\in M_n(S)$, $A$ is the sum of two idempotent matrices and a nilpotent matrix over $S$,
as desired.\end{proof}

\begin{cor} Let $R$ be regular, and let $n\geq 2$. Then $R$ is tripotent if and only if\end{cor}
\begin{enumerate}
\item [(1)]{\it $R$ is commutative;}
\vspace{-.5mm}
\item [(2)]{\it $M_n(R)$ is trinil clean.}
\end{enumerate}
\begin{proof} $\Longrightarrow $ In light of~\cite[Theorem 1]{HT}, $R$ is commutative. By virtue of Theorem 4.2, $M_n(R)$ is trinil clean, as required.

$\Longleftarrow $ Let $M$ be a maximal ideal of $R$. Then $R/M$ is a field. By hypothesis, $M_n(R)$ is trinil clean, and then so is $M_n(R/M)$.
In light of Theorem 4.1, $R/M\cong {\Bbb Z}_2$ or ${\Bbb Z}_3$. Thus, $R$ is isomorphic to the subdirect product of ${\Bbb Z}_2$'s and ${\Bbb Z}_3$'s.
Since ${\Bbb Z}_2$ and ${\Bbb Z}_3$ are both tripotent, we then easily check that $R$ is tripotent.\end{proof}

We are ready to prove the following main theorems.

\begin{thm} Let $R$ be 2-primal strongly 2-nil-clean ring. Then $M_n(R)$ is trinil clean for all $n\in {\Bbb N}$.\end{thm}
\begin{proof} According to Theorem 2.4,
$R\cong R_1\times R_2$, where $R_1$ and $R_2$ are both strongly 2-nil-clean, $2\in J(R_1)$ and $3\in J(R_2)$.
By virtue of ~\cite[Theorem 2.11]{CS}, $R_1$ is strongy nil-clean. According ~\cite [Theorem 6.1] {KWZ}, $M_n(R_1)$ is nil-clean.
As $R_2$ is a homomorphic image of a strongly 2-nil-clean ring, $R_2$ is strongly 2-nil-clean. It follows by Corollary 3.3 that $J(R_2)$ is nil and $R_2/J(R_2)$ is tripotent.
In light of Theorem 4.2, $M_n(R_2/J(R_2))$ is trinil clean. Furthermore,
$J(R_2)\subseteq N(R_2)=P(R_2)\subseteq J(R_2)$, we get $J(R_2)=P(R_2)$. Hence, $M_n(J(R_2))=M_n(P(R_2))=P(M_n(R_2))$ is nil. Obviously, $3\in J(M_2(R_2))$.
Since $M_n\big(R_2/J(R_2)\big)$ $\cong M_n(R_2)/M_n(J(R_2))$, it follows by Theorem 2.7, that $M_n(R_2)$ is trinil clean.

Therefore $M_n(R)\cong M_n(R_1)\times M_n(R_2)$ is trinil clean, as asserted.\end{proof}

\begin{cor} Let $R$ be
a commutative trinil clean ring. Then $M_n(R)$ is trinil-clean for all $n\in {\Bbb N}$.\end{cor}

Recall that a ring $R$ is weakly nil-clean if every element in $R$ is the sum or difference of a nilpotent and an idempotent (cf. \cite{KZ}).

\begin{cor} Let $R$ be a commutative weakly nil-clean ring. Then $M_n(R)$ is trinil clean for all $n\in {\Bbb N}$.\end{cor}
\begin{proof} As every commutative weakly nil-clean ring is trinil clean 2-primal ring, we obtain the result, by Theorem 4.5.\end{proof}

\begin{exam} Let $n\in {\Bbb N}$. Then $M_n ({\Bbb Z}_{m})$ is trinil clean for all $n\in {\Bbb N}$ if and only if
$m=2^k3^l (k,l\in {\Bbb N}^+, k+l\neq 0)$.
\end{exam}
\begin{proof} $\Longrightarrow $ In view of Lemma 2.3, $6\in N({\Bbb Z}_m)$. Hence, $m=2^k3^l (k,l\in {\Bbb N}^+, k+l\neq 0)$.

$\Longleftarrow $ By hypothesis, ${\Bbb Z}_m\cong {\Bbb Z}_{2^k}\oplus {\Bbb Z}_{3^l}$. Since $J\big({\Bbb Z}_{2^k}\big)=2{\Bbb Z}_{2^k}$ and ${\Bbb Z}_{2^k}/J\big({\Bbb Z}_{2^k}\big)\cong {\Bbb Z}_2$, it follows by Corollary 3.3, ${\Bbb Z}_{2^k}$ is trinil clean. Likewise, ${\Bbb Z}_{3^l}$ is trinil clean.
This shows that ${\Bbb Z}_m$ is trinil clean. This completes the proof, by Corollary 4.5.\end{proof}

Recall that a ring $R$ is of bounded index if there exists some $n\in {\Bbb N}$ such that $x^n=0$ for all nilpotent $x\in R$.

\begin{lem} (~\cite[Lemma 6.6]{KWZ}) Let $R$ be of bounded index. If $J(R)$ is nil, then $M_n(R)$ is nil for all $n\in {\Bbb N}$.
\end{lem}

\begin{thm} Let $R$ be of bounded index. If $R$ is strongly 2-nil-clean, then $M_n(R)$ is trinil clean for all $n\in {\Bbb N}$.
\end{thm}
\begin{proof} In view of Theorem 2.4, we easily see that
$R\cong R_1\times R_2$, where $R_1$ and $R_2$ are both strongly 2-nil-clean, $2\in J(R_1)$ and $3\in J(R_2)$.
By virtue of ~\cite[Theorem 2.11]{CS}, $R_1$ is strongy nil-clean. It follows by ~\cite[Thorem 6.1]{KWZ} that $M_n(R_1)$ is nil-clean.
In light of Corollary 3.3, $J(R_2)$ is nil and $R_2/J(R_2)$ is tripotent.
Hence, $M_n(R_2/J(R_2))$ is trinil clean by Theorem 4.2. Clearly, $R_2$ is of bounded index. In terms of
Corollary 3.3 , $J(M_n(R_2))$ is nil. Clearly, $3\in J(M_n(R_2))$. As $M_n\big(R_2/J(R_2)\big)$ $\cong M_n(R_2)/M_n(J(R_2))$, it follows by
Theorem 2.7, that $M_n(R_2)$ is trinil clean. Therefore $M_n(R)\cong M_n(R_1)\times M_n(R_2)$ is trinil clean.\end{proof}

\begin{cor} Let $R$ be a ring, and let $m\in {\Bbb N}$. If $(a-a^3)^m=0$ for all $a\in R$, then $M_n(R)$ is trinil clean for all $n\in {\Bbb N}$.
\end{cor}
\begin{proof} Let $x\in J(R)$. Then $(x-x^3)^m=0$, and so $x^m=0$. This implies that $J(R)$ is nil. In light of ~\cite[Theorem A.1]{HT}, $N(R)$ forms an ideal of $R$, and so $N(R)\subseteq J(R)$. Hence, $J(R)=N(R)$ is nil. Further, $R/J(R)$ is tripotent. In light of Lemma 2.7, $R$ is strongly 2-nil-clean. If $a^k=0 (k\in {\Bbb N}$, then $1-a,1+a\in U(R)$, and so $1-a^2=(1-a)(1+a)\in U(R)$.
By hypothesis, $a^m(1-a^2)^m=0$. Hence, $a^m=0$, and so $R$ is of bounded index. This complete the proof, by
Theorem 4.9.\end{proof}

A ring $R$ is a 2-Boolean ring provided that $a^2$ is an idempotent for all $a\in R$.

\begin{cor} Let $R$ be a 2-Boolean ring. Then $M_n(R)$ is trinil-clean for all $n\in {\Bbb N}$.\end{cor}
\begin{proof} Let $a\in R$. Then $a^2=a^4$. Hence, $a^2(1-a^2)=0$. This shows that
$(1-a^2)^2a^2(1-a^2)a=0$, i.e., $(a-a^3)^3=0$. In light of Corollary 4.10, the result follows.\end{proof}

\end{document}